\documentclass[12pt,reqno]{amsart}
\usepackage{amsmath}
\usepackage{amsfonts}
\usepackage{amssymb}
\usepackage{amsthm}
\usepackage[frame,arrow,curve,matrix]{xy}
\usepackage{graphicx}
\theoremstyle{plain}

\newtheorem{prop}{Proposition}[section]
\newtheorem{lem}[prop]{Lemma}
\theoremstyle{definition}
\newtheorem*{defn*}{Definition}             % unnumbered definition
\newcommand{\defnemph}[1]{\textbf{#1}}
\newenvironment{rem}{\noindent\textsl{Remark.}}{}  % perhaps looks better than rem above?
\newtheorem*{example}{Example}
\numberwithin{equation}{section}
\newcommand{\abs}[1]{\left\vert#1\right\vert}
\newcommand{\set}[1]{\left\{#1\right\}}
\newcommand{\setc}[2]{\left\{#1 \;\left| \; #2 \right. \right\}}

\newcommand{\Complex}{\mathbb C}

\newcommand{\Natural}{\mathbb N}

\newcommand{\Id}{\boldsymbol{1}}

\newcommand{\To}{\rightarrow}

\newcounter{comment}
\newcommand{\comment}[1]{}
\newcommand{\noop}[1]{}

\renewcommand{\imath}{\mathfrak{i}}
\renewcommand{\jmath}{\mathfrak{j}}

\newcommand{\JW}[1]{f^{(#1)}}
\newcommand{\coeff}[2]{\coefficient_{\in \JW{#1}}\left(#2\right)}
\DeclareMathOperator*{\coefficient}{coeff}

\newcommand{\db}[1]{\left(\left(#1\right)\right)}

\newcommand{\shortversion}[1]{}     % the short version can probably lose some appendices as well!

\newcommand{\draftversion}[1]{}

\usepackage{hyperref}

\newcommand{\arxiv}[1]{\href{http://arxiv.org/abs/#1}{\tt arXiv:\nolinkurl{#1}}}
\newcommand{\doi}[1]{\href{http://dx.doi.org/#1}{{\tt DOI:#1}}}

\newcommand{\googlebooks}[1]{(preview at \href{http://books.google.com/books?id=#1}{google books})}

\newcommand{\xyTLFiveTwentySix}{ \xybox{(0,-14.)*\xybox{
(1.5,0);(16.5,28) **i@{-},(3,28);(12,28)  **\crv{(3,22.6)&(12,22.6)},
(6,28);(9,28)  **\crv{(6,26.2)&(9,26.2)},
(15,28);(9,0)  **\crv{(15,8.4)&(9,19.6)},
(15,0);(12,0)  **\crv{(15,1.8)&(12,1.8)},
(6,0);(3,0)  **\crv{(6,1.8)&(3,1.8)},
}}}\newcommand{\xyTLFiveTwentyFiveTimesTLFiveTwelve} { \xybox{(0,-14.)*\xybox{*\xybox{
(1.5,0);(16.5,12) **i@{-},(3,12);(9,0)  **\crv{(3,3.6)&(9,8.4)},
(6,12);(15,12)  **\crv{(6,6.6)&(15,6.6)},
(9,12);(12,12)  **\crv{(9,10.2)&(12,10.2)},
(15,0);(12,0)  **\crv{(15,1.8)&(12,1.8)},
(6,0);(3,0)  **\crv{(6,1.8)&(3,1.8)},
}!U*\xybox{
(1.5,0);(16.5,16) **i@{-},(3,16);(12,16)  **\crv{(3,10.6)&(12,10.6)},
(6,16);(9,16)  **\crv{(6,14.2)&(9,14.2)},
(15,16);(3,0)  **\crv{(15,4.8)&(3,11.2)},
(15,0);(12,0)  **\crv{(15,1.8)&(12,1.8)},
(9,0);(6,0)  **\crv{(9,1.8)&(6,1.8)},
}!U}}}\newcommand{\xyMultiplicativeGeneratorFiveOne}{ \xybox{
(1.5,0);(16.5,18) **i@{-},(3,18);(6,18)  **\crv{(3,16.2)&(6,16.2)},
(9,18);(9,0)  **\crv{(9,5.4)&(9,12.6)},
(12,18);(12,0)  **\crv{(12,5.4)&(12,12.6)},
(15,18);(15,0)  **\crv{(15,5.4)&(15,12.6)},
(6,0);(3,0)  **\crv{(6,1.8)&(3,1.8)},
}}\newcommand{\xyMultiplicativeGeneratorFiveTwo}{ \xybox{
(1.5,0);(16.5,18) **i@{-},(3,18);(3,0)  **\crv{(3,5.4)&(3,12.6)},
(6,18);(9,18)  **\crv{(6,16.2)&(9,16.2)},
(12,18);(12,0)  **\crv{(12,5.4)&(12,12.6)},
(15,18);(15,0)  **\crv{(15,5.4)&(15,12.6)},
(9,0);(6,0)  **\crv{(9,1.8)&(6,1.8)},
}}\newcommand{\xyMultiplicativeGeneratorFiveThree}{ \xybox{
(1.5,0);(16.5,18) **i@{-},(3,18);(3,0)  **\crv{(3,5.4)&(3,12.6)},
(6,18);(6,0)  **\crv{(6,5.4)&(6,12.6)},
(9,18);(12,18)  **\crv{(9,16.2)&(12,16.2)},
(15,18);(15,0)  **\crv{(15,5.4)&(15,12.6)},
(12,0);(9,0)  **\crv{(12,1.8)&(9,1.8)},
}}\newcommand{\xyMultiplicativeGeneratorFiveFour}{ \xybox{
(1.5,0);(16.5,18) **i@{-},(3,18);(3,0)  **\crv{(3,5.4)&(3,12.6)},
(6,18);(6,0)  **\crv{(6,5.4)&(6,12.6)},
(9,18);(9,0)  **\crv{(9,5.4)&(9,12.6)},
(12,18);(15,18)  **\crv{(12,16.2)&(15,16.2)},
(15,0);(12,0)  **\crv{(15,1.8)&(12,1.8)},
}}\newcommand{\xyJonesWenzlIdempotentSixPlusOne}[1]{\xybox{(0,-9.)*\xybox{*\xybox{
(1.5,0);(19.5,18) **i@{-},(3,18);(3,14.4) **@{-}, (3,3.6);(3,0) **@{-},(6,18);(6,14.4) **@{-}, (6,3.6);(6,0) **@{-},(9,18);(9,14.4) **@{-}, (9,3.6);(9,0) **@{-},(15,18);(15,14.4) **@{-}, (15,3.6);(15,0) **@{-},(18,18);(18,14.4) **@{-}, (18,3.6);(18,0) **@{-},(12,16.2) *h+++{\cdots},(12,1.8) *+++{\cdots},(2.25,3.6);(18.75,14.4) **\frm{-}, (10.5,9.) *+++{#1}}!R!(-1.5,0)*\xybox{
(1.5,0);(4.5,18) **i@{-},(3,18);(3,0)  **\crv{(3,5.4)&(3,12.6)},
}!R!(-1.5,0)}}}\newcommand{\xyJonesWenzlIdempotentSeven}[1]{\xybox{(0,-9.)*\xybox{
(1.5,0);(22.5,18) **i@{-},(3,18);(3,14.4) **@{-}, (3,3.6);(3,0) **@{-},(6,18);(6,14.4) **@{-}, (6,3.6);(6,0) **@{-},(9,18);(9,14.4) **@{-}, (9,3.6);(9,0) **@{-},(15,18);(15,14.4) **@{-}, (15,3.6);(15,0) **@{-},(18,18);(18,14.4) **@{-}, (18,3.6);(18,0) **@{-},(21,18);(21,14.4) **@{-}, (21,3.6);(21,0) **@{-},(12,16.2) *h+++{\cdots},(12,1.8) *+++{\cdots},(2.25,3.6);(21.75,14.4) **\frm{-}, (12.,9.) *+++{#1}}}}\newcommand{\xyWenzlRecurrenceLastTermSixPlusOne}[1]{\xybox{(0,-20.)*\xybox{*\xybox{*\xybox{
(1.5,0);(19.5,18) **i@{-},(3,18);(3,14.4) **@{-}, (3,3.6);(3,0) **@{-},(6,18);(6,14.4) **@{-}, (6,3.6);(6,0) **@{-},(9,18);(9,14.4) **@{-}, (9,3.6);(9,0) **@{-},(15,18);(15,14.4) **@{-}, (15,3.6);(15,0) **@{-},(18,18);(18,14.4) **@{-}, (18,3.6);(18,0) **@{-},(12,16.2) *h+++{\cdots},(12,1.8) *+++{\cdots},(2.25,3.6);(18.75,14.4) **\frm{-}, (10.5,9.) *+++{#1}}!R!(-1.5,0)*\xybox{
(1.5,0);(4.5,18) **i@{-},(3,18);(3,0)  **\crv{(3,5.4)&(3,12.6)},
}!R!(-1.5,0)}!U*\xybox{*\xybox{
(1.5,0);(10.5,4) **i@{-},(3,4);(3,0)  **\crv{(3,1.2)&(3,2.8)},
(6,4);(6,0)  **\crv{(6,1.2)&(6,2.8)},
(9,4);(9,0)  **\crv{(9,1.2)&(9,2.8)},
}!R!(-1.5,0)*\xybox{*i\xybox{
(1.5,0);(4.5,4) **i@{-},(3,4);(3,0)  **\crv{(3,1.2)&(3,2.8)},
}}!R!(-1.5,0)*\xybox{
(1.5,0);(10.5,4) **i@{-},(3,4);(3,0)  **\crv{(3,1.2)&(3,2.8)},
(6,4);(9,4)  **\crv{(6,2.2)&(9,2.2)},
(9,0);(6,0)  **\crv{(9,1.8)&(6,1.8)},
}!R!(-1.5,0)}!U*\xybox{*\xybox{
(1.5,0);(19.5,18) **i@{-},(3,18);(3,14.4) **@{-}, (3,3.6);(3,0) **@{-},(6,18);(6,14.4) **@{-}, (6,3.6);(6,0) **@{-},(9,18);(9,14.4) **@{-}, (9,3.6);(9,0) **@{-},(15,18);(15,14.4) **@{-}, (15,3.6);(15,0) **@{-},(18,18);(18,14.4) **@{-}, (18,3.6);(18,0) **@{-},(12,16.2) *h+++{\cdots},(12,1.8) *+++{\cdots},(2.25,3.6);(18.75,14.4) **\frm{-}, (10.5,9.) *+++{#1}}!R!(-1.5,0)*\xybox{
(1.5,0);(4.5,18) **i@{-},(3,18);(3,0)  **\crv{(3,5.4)&(3,12.6)},
}!R!(-1.5,0)}!U}}}\newcommand{\xySingleRightCupSixOne}{\xybox{
(1.5,0);(19.5,18) **i@{-},(3,18);(9,0)  **\crv{(3,5.4)&(9,12.6)},
(6,18);(12,0)  **\crv{(6,5.4)&(12,12.6)},
(9,18);(15,0)  **\crv{(9,5.4)&(15,12.6)},
(12,18);(18,0)  **\crv{(12,5.4)&(18,12.6)},
(15,18);(18,18)  **\crv{(15,16.2)&(18,16.2)},
(6,0);(3,0)  **\crv{(6,1.8)&(3,1.8)},
}}\newcommand{\xySingleRightCupSixTwo}{\xybox{
(1.5,0);(19.5,18) **i@{-},(3,18);(3,0)  **\crv{(3,5.4)&(3,12.6)},
(6,18);(12,0)  **\crv{(6,5.4)&(12,12.6)},
(9,18);(15,0)  **\crv{(9,5.4)&(15,12.6)},
(12,18);(18,0)  **\crv{(12,5.4)&(18,12.6)},
(15,18);(18,18)  **\crv{(15,16.2)&(18,16.2)},
(9,0);(6,0)  **\crv{(9,1.8)&(6,1.8)},
}}\newcommand{\xySingleRightCupSixThree}{\xybox{
(1.5,0);(19.5,18) **i@{-},(3,18);(3,0)  **\crv{(3,5.4)&(3,12.6)},
(6,18);(6,0)  **\crv{(6,5.4)&(6,12.6)},
(9,18);(15,0)  **\crv{(9,5.4)&(15,12.6)},
(12,18);(18,0)  **\crv{(12,5.4)&(18,12.6)},
(15,18);(18,18)  **\crv{(15,16.2)&(18,16.2)},
(12,0);(9,0)  **\crv{(12,1.8)&(9,1.8)},
}}\newcommand{\xySingleRightCupSixFour}{\xybox{
(1.5,0);(19.5,18) **i@{-},(3,18);(3,0)  **\crv{(3,5.4)&(3,12.6)},
(6,18);(6,0)  **\crv{(6,5.4)&(6,12.6)},
(9,18);(9,0)  **\crv{(9,5.4)&(9,12.6)},
(12,18);(18,0)  **\crv{(12,5.4)&(18,12.6)},
(15,18);(18,18)  **\crv{(15,16.2)&(18,16.2)},
(15,0);(12,0)  **\crv{(15,1.8)&(12,1.8)},
}}\newcommand{\xySingleRightCupSixFive}{\xybox{
(1.5,0);(19.5,18) **i@{-},(3,18);(3,0)  **\crv{(3,5.4)&(3,12.6)},
(6,18);(6,0)  **\crv{(6,5.4)&(6,12.6)},
(9,18);(9,0)  **\crv{(9,5.4)&(9,12.6)},
(12,18);(12,0)  **\crv{(12,5.4)&(12,12.6)},
(15,18);(18,18)  **\crv{(15,16.2)&(18,16.2)},
(18,0);(15,0)  **\crv{(18,1.8)&(15,1.8)},
}}\newcommand{\xySingleRightCupSixSix}{\xybox{
(1.5,0);(19.5,18) **i@{-},(3,18);(3,0)  **\crv{(3,5.4)&(3,12.6)},
(6,18);(6,0)  **\crv{(6,5.4)&(6,12.6)},
(9,18);(9,0)  **\crv{(9,5.4)&(9,12.6)},
(12,18);(12,0)  **\crv{(12,5.4)&(12,12.6)},
(15,18);(15,0)  **\crv{(15,5.4)&(15,12.6)},
(18,18);(18,0)  **\crv{(18,5.4)&(18,12.6)},
}}\newcommand{\xyMultiplicativeGeneratorSixFiveTimesGFiveThree}{\xybox{(0,-9.)*\xybox{*\xybox{*\xybox{
(1.5,0);(16.5,9) **i@{-},(3,9);(3,0)  **\crv{(3,2.7)&(3,6.3)},
(6,9);(6,0)  **\crv{(6,2.7)&(6,6.3)},
(9,9);(15,0)  **\crv{(9,2.7)&(15,6.3)},
(12,9);(15,9)  **\crv{(12,7.2)&(15,7.2)},
(12,0);(9,0)  **\crv{(12,1.8)&(9,1.8)},
}!R!(-1.5,0)*\xybox{
(1.5,0);(4.5,9) **i@{-},(3,9);(3,0)  **\crv{(3,2.7)&(3,6.3)},
}!R!(-1.5,0)}!U*\xybox{
(1.5,0);(19.5,9) **i@{-},(3,9);(3,0)  **\crv{(3,2.7)&(3,6.3)},
(6,9);(6,0)  **\crv{(6,2.7)&(6,6.3)},
(9,9);(9,0)  **\crv{(9,2.7)&(9,6.3)},
(12,9);(12,0)  **\crv{(12,2.7)&(12,6.3)},
(15,9);(18,9)  **\crv{(15,7.2)&(18,7.2)},
(18,0);(15,0)  **\crv{(18,1.8)&(15,1.8)},
}!U}}}\newcommand{\xyGSixThree}{\xybox{(0,-9.)*\xybox{
(1.5,0);(19.5,18) **i@{-},(3,18);(3,0)  **\crv{(3,5.4)&(3,12.6)},
(6,18);(6,0)  **\crv{(6,5.4)&(6,12.6)},
(9,18);(15,0)  **\crv{(9,5.4)&(15,12.6)},
(12,18);(18,0)  **\crv{(12,5.4)&(18,12.6)},
(15,18);(18,18)  **\crv{(15,16.2)&(18,16.2)},
(12,0);(9,0)  **\crv{(12,1.8)&(9,1.8)},
}}}\newcommand{\xyCaseOne}{\xybox{(0,-12.)*\xybox{*\xybox{
(1.5,0);(16.5,12) **i@{-},(3,12);(3,0)  **\crv{(3,3.6)&(3,8.4)},
(6,12);(12,0)  **\crv{(6,3.6)&(12,8.4)},
(9,12);(15,0)  **\crv{(9,3.6)&(15,8.4)},
(12,12);(15,12)  **\crv{(12,10.2)&(15,10.2)},
(9,0);(6,0)  **\crv{(9,1.8)&(6,1.8)},
}!U*\xybox{*\xybox{*\xybox{
(1.5,0);(13.5,12) **i@{-},(3,12);(3,9.6) **@{-}, (3,2.4);(3,0) **@{-},(6,12);(6,9.6) **@{-}, (6,2.4);(6,0) **@{-},(9,12);(9,9.6) **@{-}, (9,2.4);(9,0) **@{-},(12,12);(12,9.6) **@{-}, (12,2.4);(12,0) **@{-},(2.25,2.4);(12.75,9.6) **\frm{-}, (7.5,6.) *+++{h}},(6,2.4);(9,2.4) **h\crv{(6,4)&(9,4)}}!R!(-1.5,0)*\xybox{
(1.5,0);(4.5,12) **i@{-},(3,12);(3,0)  **\crv{(3,3.6)&(3,8.4)},
}!R!(-1.5,0)}!U}}}\newcommand{\xyCaseTwo}{\xybox{(0,-12.)*\xybox{*\xybox{
(1.5,0);(16.5,12) **i@{-},(3,12);(3,0)  **\crv{(3,3.6)&(3,8.4)},
(6,12);(12,0)  **\crv{(6,3.6)&(12,8.4)},
(9,12);(15,0)  **\crv{(9,3.6)&(15,8.4)},
(12,12);(15,12)  **\crv{(12,10.2)&(15,10.2)},
(9,0);(6,0)  **\crv{(9,1.8)&(6,1.8)},
}!U*\xybox{*\xybox{*\xybox{
(1.5,0);(13.5,12) **i@{-},(3,12);(3,9.6) **@{-}, (3,2.4);(3,0) **@{-},(6,12);(6,9.6) **@{-}, (6,2.4);(6,0) **@{-},(9,12);(9,9.6) **@{-}, (9,2.4);(9,0) **@{-},(12,12);(12,9.6) **@{-}, (12,2.4);(12,0) **@{-},(2.25,2.4);(12.75,9.6) **\frm{-}, (7.5,6.) *+++{h}},(9,2.4);(12,2.4) **h\crv{(9,4)&(12,4)}}!R!(-1.5,0)*\xybox{
(1.5,0);(4.5,12) **i@{-},(3,12);(3,0)  **\crv{(3,3.6)&(3,8.4)},
}!R!(-1.5,0)}!U}}}\newcommand{\xyTLSixSeventyNine}{ \xybox{(0,-9.)*\xybox{
(1.5,0);(19.5,18) **i@{-},(3,18);(12,18)  **\crv{(3,12.6)&(12,12.6)},
(6,18);(9,18)  **\crv{(6,16.2)&(9,16.2)},
(15,18);(3,0)  **\crv{(15,5.4)&(3,12.6)},
(18,18);(12,0)  **\crv{(18,5.4)&(12,12.6)},
(18,0);(15,0)  **\crv{(18,1.8)&(15,1.8)},
(9,0);(6,0)  **\crv{(9,1.8)&(6,1.8)},
}}}\newcommand{\xyTLSixEightyOneAndCap}{ \xybox{(0,-9.)*\xybox{*\xybox{
(1.5,0);(19.5,18) **i@{-},(3,18);(12,18)  **\crv{(3,12.6)&(12,12.6)},
(6,18);(9,18)  **\crv{(6,16.2)&(9,16.2)},
(15,18);(3,0)  **\crv{(15,5.4)&(3,12.6)},
(18,18);(18,0)  **\crv{(18,5.4)&(18,12.6)},
(15,0);(6,0)  **\crv{(15,5.4)&(6,5.4)},
(12,0);(9,0)  **\crv{(12,1.8)&(9,1.8)},
}!R!(-1.5,0)*\xybox{
(1.5,0);(7.5,18) **i@{-},(6,0);(3,0)  **\crv{(6,1.8)&(3,1.8)},
}!R!(-1.5,0)}}}\newcommand{\xyTLSixSeventyNineTimesGSevenThree}{\xybox{(0,-13.)*\xybox{*\xybox{
(1.5,0);(22.5,13) **i@{-},(3,13);(3,0)  **\crv{(3,3.9)&(3,9.1)},
(6,13);(6,0)  **\crv{(6,3.9)&(6,9.1)},
(9,13);(15,0)  **\crv{(9,3.9)&(15,9.1)},
(12,13);(18,0)  **\crv{(12,3.9)&(18,9.1)},
(15,13);(21,0)  **\crv{(15,3.9)&(21,9.1)},
(18,13);(21,13)  **\crv{(18,11.2)&(21,11.2)},
(12,0);(9,0)  **\crv{(12,1.8)&(9,1.8)},
}!U*\xybox{*\xybox{
(1.5,0);(19.5,13) **i@{-},(3,13);(12,13)  **\crv{(3,7.6)&(12,7.6)},
(6,13);(9,13)  **\crv{(6,11.2)&(9,11.2)},
(15,13);(3,0)  **\crv{(15,3.9)&(3,9.1)},
(18,13);(12,0)  **\crv{(18,3.9)&(12,9.1)},
(18,0);(15,0)  **\crv{(18,1.8)&(15,1.8)},
(9,0);(6,0)  **\crv{(9,1.8)&(6,1.8)},
}!R!(-1.5,0)*\xybox{
(1.5,0);(4.5,13) **i@{-},(3,13);(3,0)  **\crv{(3,3.9)&(3,9.1)},
}!R!(-1.5,0)}!U}}}\newcommand{\xyTLSevenTwoHundredAndSixtyFive}{ \xybox{(0,-9.)*\xybox{
(1.5,0);(22.5,18) **i@{-},(3,18);(12,18)  **\crv{(3,12.6)&(12,12.6)},
(6,18);(9,18)  **\crv{(6,16.2)&(9,16.2)},
(15,18);(3,0)  **\crv{(15,5.4)&(3,12.6)},
(18,18);(18,0)  **\crv{(18,5.4)&(18,12.6)},
(21,18);(21,0)  **\crv{(21,5.4)&(21,12.6)},
(15,0);(6,0)  **\crv{(15,5.4)&(6,5.4)},
(12,0);(9,0)  **\crv{(12,1.8)&(9,1.8)},
}}}\newcommand{\xyTLFiveEight}{ \xybox{(0,-4.)*\xybox{
(0.75,0);(8.25,8) **i@{-},(1.5,8);(1.5,0)  **\crv{(1.5,2.4)&(1.5,5.6)},
(3.,8);(6.,0)  **\crv{(3.,2.4)&(6.,5.6)},
(4.5,8);(6.,8)  **\crv{(4.5,7.1)&(6.,7.1)},
(7.5,8);(7.5,0)  **\crv{(7.5,2.4)&(7.5,5.6)},
(4.5,0);(3.,0)  **\crv{(4.5,0.9)&(3.,0.9)},
}}}\newcommand{\xyFoldDownTLFiveEight}{ \xybox{(0,-4.)*\xybox{*\xybox{
(0.75,0);(6.75,8) **i@{-},(1.5,8);(1.5,0)  **\crv{(1.5,2.4)&(1.5,5.6)},
(3.,8);(6.,0)  **\crv{(3.,2.4)&(6.,5.6)},
(4.5,8);(6.,8)  **\crv{(4.5,7.1)&(6.,7.1)},
(4.5,0);(3.,0)  **\crv{(4.5,0.9)&(3.,0.9)},
}!R!(-1.5,0)*\xybox{
(0.75,0);(3.75,8) **i@{-},(3.,0);(1.5,0)  **\crv{(3.,0.9)&(1.5,0.9)},
}!R!(-1.5,0)}}}\newcommand{\xyTLFourTwo}{ \xybox{(0,-4.)*\xybox{
(0.75,0);(6.75,8) **i@{-},(1.5,8);(1.5,0)  **\crv{(1.5,2.4)&(1.5,5.6)},
(3.,8);(3.,0)  **\crv{(3.,2.4)&(3.,5.6)},
(4.5,8);(6.,8)  **\crv{(4.5,7.1)&(6.,7.1)},
(6.,0);(4.5,0)  **\crv{(6.,0.9)&(4.5,0.9)},
}}}\newcommand{\xyTLFourThree}{ \xybox{(0,-4.)*\xybox{
(0.75,0);(6.75,8) **i@{-},(1.5,8);(1.5,0)  **\crv{(1.5,2.4)&(1.5,5.6)},
(3.,8);(6.,0)  **\crv{(3.,2.4)&(6.,5.6)},
(4.5,8);(6.,8)  **\crv{(4.5,7.1)&(6.,7.1)},
(4.5,0);(3.,0)  **\crv{(4.5,0.9)&(3.,0.9)},
}}}\newcommand{\xyTLThreeOne}{ \xybox{(0,-4.)*\xybox{
(0.75,0);(5.25,8) **i@{-},(1.5,8);(1.5,0)  **\crv{(1.5,2.4)&(1.5,5.6)},
(3.,8);(3.,0)  **\crv{(3.,2.4)&(3.,5.6)},
(4.5,8);(4.5,0)  **\crv{(4.5,2.4)&(4.5,5.6)},
}}}\newcommand{\xyFoldDownTLFourTwo}{ \xybox{(0,-4.)*\xybox{
(0.75,0);(8.25,8) **i@{-},(1.5,8);(1.5,0)  **\crv{(1.5,2.4)&(1.5,5.6)},
(3.,8);(3.,0)  **\crv{(3.,2.4)&(3.,5.6)},
(4.5,8);(7.5,0)  **\crv{(4.5,2.4)&(7.5,5.6)},
(6.,0);(4.5,0)  **\crv{(6.,0.9)&(4.5,0.9)},
}}}\newcommand{\xyFoldDownTLFourThree}{ \xybox{(0,-4.)*\xybox{
(0.75,0);(8.25,8) **i@{-},(1.5,8);(1.5,0)  **\crv{(1.5,2.4)&(1.5,5.6)},
(3.,8);(6.,0)  **\crv{(3.,2.4)&(6.,5.6)},
(4.5,8);(7.5,0)  **\crv{(4.5,2.4)&(7.5,5.6)},
(4.5,0);(3.,0)  **\crv{(4.5,0.9)&(3.,0.9)},
}}}\newcommand{\xyTLFiveEightFullSize}{ \xybox{(0,-9.)*\xybox{
(1.5,0);(16.5,18) **i@{-},(3,18);(3,0)  **\crv{(3,5.4)&(3,12.6)},
(6,18);(12,0)  **\crv{(6,5.4)&(12,12.6)},
(9,18);(12,18)  **\crv{(9,16.2)&(12,16.2)},
(15,18);(15,0)  **\crv{(15,5.4)&(15,12.6)},
(9,0);(6,0)  **\crv{(9,1.8)&(6,1.8)},
}}}\newcommand{\xyCapFormTLFiveEight}{ \xybox{(0,-9.)*\xybox{
(1.5,0);(31.5,18) **i@{-},(30,0);(3,0)  **\crv{(30,16.2)&(3,16.2)},
(27,0);(12,0)  **\crv{(27,9.)&(12,9.)},
(24,0);(21,0)  **\crv{(24,1.8)&(21,1.8)},
(18,0);(15,0)  **\crv{(18,1.8)&(15,1.8)},
(9,0);(6,0)  **\crv{(9,1.8)&(6,1.8)},
}}}\newcommand{\xyTLMoveInitial}{ \xybox{(0,-11.)*\xybox{
(1.5,0);(31.5,22) **i@{-},(3,22);(6,22)  **\crv{(3,20.2)&(6,20.2)},
(9,22);(12,22)  **\crv{(9,20.2)&(12,20.2)},
(15,22);(18,22)  **\crv{(15,20.2)&(18,20.2)},
(21,22);(24,22)  **\crv{(21,20.2)&(24,20.2)},
(27,22);(3,0)  **\crv{(27,6.6)&(3,15.4)},
(30,22);(30,0)  **\crv{(30,6.6)&(30,15.4)},
(27,0);(6,0)  **\crv{(27,12.6)&(6,12.6)},
(24,0);(9,0)  **\crv{(24,9.)&(9,9.)},
(21,0);(12,0)  **\crv{(21,5.4)&(12,5.4)},
(18,0);(15,0)  **\crv{(18,1.8)&(15,1.8)},
}}}\newcommand{\xyTLMoveFinalOne}{ \xybox{(0,-11.)*\xybox{
(1.5,0);(31.5,22) **i@{-},(3,22);(6,22)  **\crv{(3,20.2)&(6,20.2)},
(9,22);(12,22)  **\crv{(9,20.2)&(12,20.2)},
(15,22);(18,22)  **\crv{(15,20.2)&(18,20.2)},
(21,22);(24,22)  **\crv{(21,20.2)&(24,20.2)},
(27,22);(3,0)  **\crv{(27,6.6)&(3,15.4)},
(30,22);(30,0)  **\crv{(30,6.6)&(30,15.4)},
(27,0);(12,0)  **\crv{(27,9.)&(12,9.)},
(24,0);(15,0)  **\crv{(24,5.4)&(15,5.4)},
(21,0);(18,0)  **\crv{(21,1.8)&(18,1.8)},
(9,0);(6,0)  **\crv{(9,1.8)&(6,1.8)},
}}}\newcommand{\xyTLThreeOneFullSize}{ \xybox{(0,-9.)*\xybox{
(1.5,0);(10.5,18) **i@{-},(3,18);(3,0)  **\crv{(3,5.4)&(3,12.6)},
(6,18);(6,0)  **\crv{(6,5.4)&(6,12.6)},
(9,18);(9,0)  **\crv{(9,5.4)&(9,12.6)},
}}}\newcommand{\xyTLThreeTwoFullSize}{ \xybox{(0,-9.)*\xybox{
(1.5,0);(10.5,18) **i@{-},(3,18);(3,0)  **\crv{(3,5.4)&(3,12.6)},
(6,18);(9,18)  **\crv{(6,16.2)&(9,16.2)},
(9,0);(6,0)  **\crv{(9,1.8)&(6,1.8)},
}}}\newcommand{\xyTLThreeThreeFullSize}{ \xybox{(0,-9.)*\xybox{
(1.5,0);(10.5,18) **i@{-},(3,18);(9,0)  **\crv{(3,5.4)&(9,12.6)},
(6,18);(9,18)  **\crv{(6,16.2)&(9,16.2)},
(6,0);(3,0)  **\crv{(6,1.8)&(3,1.8)},
}}}\newcommand{\xyTLThreeFourFullSize}{ \xybox{(0,-9.)*\xybox{
(1.5,0);(10.5,18) **i@{-},(3,18);(6,18)  **\crv{(3,16.2)&(6,16.2)},
(9,18);(3,0)  **\crv{(9,5.4)&(3,12.6)},
(9,0);(6,0)  **\crv{(9,1.8)&(6,1.8)},
}}}\newcommand{\xyTLThreeFiveFullSize}{ \xybox{(0,-9.)*\xybox{
(1.5,0);(10.5,18) **i@{-},(3,18);(6,18)  **\crv{(3,16.2)&(6,16.2)},
(9,18);(9,0)  **\crv{(9,5.4)&(9,12.6)},
(6,0);(3,0)  **\crv{(6,1.8)&(3,1.8)},
}}}

\author{Scott Morrison}
\title{A Formula for the Jones-Wenzl Projections}
\date{}

\begin{document}

%\include{Introduction}

%% remember; writing for the arxiv.
\begin{abstract}
I present a method of calculating the coefficients appearing in the
Jones-Wenzl projections in the Temperley-Lieb algebras. It
essentially repeats the approach of Frenkel and Khovanov in
\cite{MR1446615} published in 1997. I wrote this note mid-2002,
not knowing about their work, but then set it aside upon discovering
their article.

Recently I decided to dust it off and place it on the arXiv ---
hoping the self-contained and detailed proof I give here may be
useful. It's also been cited a number of times,
%\cite{1102.2052,0902.1294,MR2873427,MR2979509,1301.1733,1302.5148,1308.2369,1308.5197,1312.5254,1501.04672,1502.06543},
so I thought it best to give it a permanent home.

The proof is based upon a simplification of the Wenzl recurrence
relation. I give an example calculation, and compare this method to
the formula announced by Ocneanu \cite{MR1907188} and partially
proved by Reznikoff \cite{MR2375712}. I also describe certain
moves on diagrams which modify their coefficients in a simple way.
\end{abstract}

\maketitle

\section{Basic Definitions}

The \defnemph{quantum integers} are denoted by $[n]$, and are
given in terms of the formal \defnemph{quantum parameter} $q$ by
the formula
\[[n] = q^{n-1} + q^{n-3} + \cdots + q^{-(n-1)} =
\frac{q^n-q^{-n}}{q-q^{-1}}.\]

The quantum integers satisfy many relations, all of which reduce
to simple arithmetic relations when evaluated at $q=1$. For
example, a simple result we will need later is
\begin{lem}\label{lem:quantum-integer-lemma}
If $m \geq a$, then $[m-a]+[m+1][a] = [m][a+1]$.
\end{lem}

An \defnemph{$n$ strand Temperley-Lieb diagram} is a diagram drawn
inside a rectangle with $n$ marked points on both the upper and
lower edges, with non-intersecting arcs joining these points. We
consider isotopic diagrams as equivalent. A \defnemph{through
strand} is an arc joining a point on the upper edge of a diagram
to the lower edge. A \defnemph{cup} joins a point on the upper
edge with another point on the upper edge, and similarly a
\defnemph{cap} joins the lower edge to itself. A cap or cup is
called \defnemph{innermost} if it is exactly that --- there are no
nested caps or cups inside it. This terminology is illustrated in
Figure \ref{fig:diagram-terminology}.

\begin{figure}[!hbp]\label{fig:diagram-terminology}
\[\xybox{ (1.5,0);(28.5,18) **i@{-},(3,18);(3,0)
**\crv{(3,5.4)&(3,12.6)}, (6,18);(12,0) **\crv{(6,5.4)&(12,12.6)},
(9,18);(15,0) **\crv{(9,5.4)&(15,12.6)}, (12,18);(21,18)
**\crv{(12,12.6)&(21,12.6)}, (15,18);(18,18)
**\crv{(15,16.2)&(18,16.2)}, (24,18);(27,18)
**\crv{(24,16.2)&(27,16.2)}, (27,0);(18,0)
**\crv{(27,5.4)&(18,5.4)}, (24,0);(21,0)
**\crv{(24,1.8)&(21,1.8)}, (9,0);(6,0)  **\crv{(9,1.8)&(6,1.8)},
\POS(-9,4)*{\txt{\small through\\strands}} \ar @{-->}@/_/+(11,5)
\ar @{-->}@/_/+(17,5) \ar @{-->}@/_/+(20,5),
\POS(20,-6)*{\txt{\small a non-innermost cap}}
\ar@{-->}@/^/+(-2,4), \POS(40,11)*{\txt{\small innermost\\cups}}
\ar@{-->}@/^/+(-22,5) \ar@{-->}@/_/+(-11,7)}\]\caption{}
\end{figure}

\noop{\xybox{ (2.5,0);(47.5,30) **i@{-},(5,30);(5,0)
**\crv{(5,9.)&(5,21.)}, (10,30);(20,0)  **\crv{(10,9.)&(20,21.)},
(15,30);(25,0)  **\crv{(15,9.)&(25,21.)}, (20,30);(35,30)
**\crv{(20,21.)&(35,21.)}, (25,30);(30,30)
**\crv{(25,27.)&(30,27.)}, (40,30);(45,30)
**\crv{(40,27.)&(45,27.)}, (45,0);(30,0)  **\crv{(45,9.)&(30,9.)},
(40,0);(35,0)  **\crv{(40,3.)&(35,3.)}, (15,0);(10,0)
**\crv{(15,3.)&(10,3.)}, \POS(-15,7)*{\txt{through\\strands}} \ar
@{-->}@/_/+(19,7) \ar @{-->}@/_/+(29,7) \ar @{-->}@/_/+(34,7),
\POS(32,-10)*{\txt{a non-innermost cap}} \ar@{-->}@/^/+(-2,8),
\POS(30,44)*{\txt{innermost\\cups}} \ar@{-->}@/_/+(0,-12)
\ar@{-->}@/^/+(10,-12)}}

The \defnemph{$n$ strand Temperley-Lieb algebra}, denoted $TL_n$,
is the algebra over $\Complex(q)$ spanned by the Temperley-Lieb
diagrams, with multiplication defined on this basis by stacking
diagrams. In such a product of diagrams closed loops may appear,
each of which we remove while inserting an additional factor of
$-[2]$. Two quite different sign conventions appear in the
literature. Generally, in topological applications loops are given
the value $-[2]$, but in the theory of subfactors the value $[2]$.
I have employed the present convention, because it results in
simpler formulas, with all coefficients positive. To pass between
the two conventions, replace everywhere $[i]$ with $(-1)^{i+1}
[i]$, or equivalently $q$ with $-q$.

Figure \ref{fig:TL-multiplication} illustrates multiplication in
the $5$ strand algebra.

\begin{figure}[!hbp]
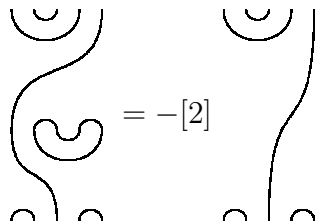

\[
\xyTLFiveTwentyFiveTimesTLFiveTwelve = - [2] \xyTLFiveTwentySix
\]
\caption{A calculation in the $5$ strand Temperley-Lieb algebra.}
\label{fig:TL-multiplication}
\end{figure}

We can also define vector spaces $TL_{n,m}$, spanned by isotopy
classes of diagrams with $m$ points on the lower boundary of the
rectangle, and $n$ along the top. These fit together into a
monoidal category \cite{MR1366832,MR1300632} over
$\Complex(q)$, with objects in $\Natural$, and $TL_{n,m}$ giving
the morphisms from $m$ to $n$.

Equivalently, we can give a definition of the Temperley-Lieb
algebra in terms of generators and relations
\cite{MR1134131}.\footnote{For the relationship between the
diagrammatic algebra and `generators and relations' algebra when
the formal parameter $q$ has been evaluated at a complex root of
unity, see \cite{MR1961959,MR2274519}.} Define the
\defnemph{multiplicative generator} $e_i$ ($i=1,\ldots,n-1$) as the diagram with
$i-1$ vertical strands, a cap-cup pair, then $n-i-1$ more vertical
strands. Figure \ref{fig:multiplicative-generators} illustrates
the multiplicative generators in the $5$ strand algebra.

\begin{figure}[!hbp]
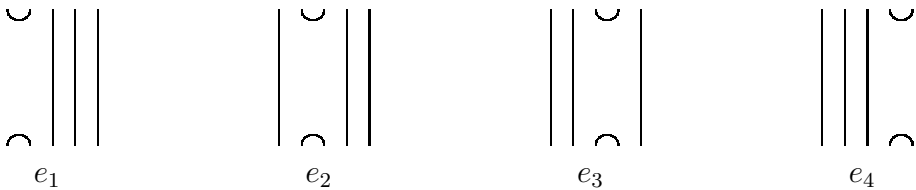

\center
\begin{tabular}{ccccccc}
\mbox{\xyMultiplicativeGeneratorFiveOne } & \hspace{0.5in} &
\mbox{\xyMultiplicativeGeneratorFiveTwo } & \hspace{0.5in} &
\mbox{\xyMultiplicativeGeneratorFiveThree } & \hspace{0.5in} &
\mbox{\xyMultiplicativeGeneratorFiveFour }
\\
$e_1$ & & $e_2$ & & $e_3$ & & $e_4$
\end{tabular}
\caption{The multiplicative generators in the $5$ strand
Temperley-Lieb algebra.} \label{fig:multiplicative-generators}
\end{figure}

The Temperley-Lieb algebra is generated by these diagrams along
with the identity diagram, denoted $\Id$, subject to the relations
\begin{align*}
e_i e_i & = -[2] e_i \\
e_i e_{i \pm 1} e_i & = e_i \\
e_i e_j & = e_j e_i \qquad \text{if $\abs{i-j} \geq 2$.}
\end{align*}

Inside the Temperley-Lieb algebra $TL_n$ we have the two-sided
ideal $\mathcal{I}_n$, generated by the elements
$\set{e_1,\ldots,e_{n-1}}$. This ideal has codimension $1$; it is
spanned by diagrams with $n-2$ or fewer through strands, that is,
every diagram except the identity diagram.

\section{The Jones-Wenzl idempotent}

Inside the $n$ strand Temperley-Lieb algebra there is a special
element called the \defnemph{Jones-Wenzl idempotent}, denoted
$\JW{n}$. It is characterised by the properties
\begin{equation}
\label{eq:JW-characterisation}
\begin{split}
 \JW{n} & \neq 0 \\
 \JW{n} \JW{n} & = \JW{n} \\
 e_i \JW{n} = \JW{n} e_i & = 0 \qquad \forall i \in \set{1,\dots,n-1}.
\end{split}
\end{equation}

The second equation could be equivalently stated as $\mathcal{I}_n
\JW{n} = \JW{n} \mathcal{I}_n = 0$.

The aim of this work is to present new methods for calculating the
coefficients for each diagram appearing in the Jones-Wenzl
idempotent. The starting point will be the Wenzl recurrence
formula, allowing us to calculate $\JW{n+1}$ in terms of $\JW{n}$.

\begin{lem}
The coefficient of the identity diagram in a Jones-Wenzl
idempotent is always $1$.
\end{lem}
\begin{proof}
Write $\JW{n} = \alpha \Id + g$, with $\alpha \in \Complex$ and $g
\in \mathcal{I}_n$. We want to see that $\alpha = 1$. This follows
from $\JW{n} \JW{n} = \JW{n} \alpha \Id + \JW{n} g = \alpha \JW{n}
+ 0$, so $\alpha = 1$.
\end{proof}

\begin{lem}
The Jones-Wenzl idempotent, characterised by Equation
\ref{eq:JW-characterisation}, is unique.
\end{lem}
\begin{proof}
Suppose both $f_1^{n}$ and $f_2^{n}$ satisfy Equation
\ref{eq:JW-characterisation}. Write $f_1^{n} = \Id + g_1$ and
$f_2^{n}=\Id+g_2$, where $g_1,g_2 \in \mathcal{I}_n$. Then
$f_1^{n} f_2^{n} = f_1^{n} (\Id+g_2) = f_1^{n}$ and similarly
$f_1^{n} f_2^{n} = (\Id+g_1)f_2^{n} = f_2^{n}$.
Thus $f_1^{n} = f_2^{n}$.
\end{proof}

For example, the $3$ strand idempotent is

\begin{equation*}
\JW{3} = \xyTLThreeOneFullSize + \frac{[2]}{[3]}
\xyTLThreeTwoFullSize + \frac{[2]}{[3]} \xyTLThreeFiveFullSize +
\frac{1}{[3]} \xyTLThreeThreeFullSize + \frac{1}{[3]}
\xyTLThreeFourFullSize
\end{equation*}

The $n$ strand Temperley-Lieb algebra naturally includes into the
$n+1$ strand algebra, by adding a vertical strand to the right
side of the diagram. Taking advantage of this, we abuse notation
and write $\JW{n} \in TL_{n+1}$ to mean the $n$ strand Jones-Wenzl
idempotent, with a vertical strand added to the right, living in
the $n+1$ strand algebra.

\begin{prop}[Wenzl recurrence formula]
The Jones-Wenzl idempotent satisfies
\begin{equation}\label{eq:wenzl-formula}
\JW{n+1} = \JW{n} + \frac{[n]}{[n+1]}\JW{n} e_{n} \JW{n},
\end{equation}
or, diagrammatically,
\[
\xyJonesWenzlIdempotentSeven{\JW{n+1}} =
\xyJonesWenzlIdempotentSixPlusOne{\JW{n}} + \frac{[n]}{[n+1]}
\xyWenzlRecurrenceLastTermSixPlusOne{\JW{n}}.
\]
\end{prop}
This is a well known result. The original paper is
\cite{MR0873400}. Various proofs can be found in any of
\cite{MR1366832,MR1858113,MR1280463,MR1227009}.

%% some examples!

\section{Simplifications of the Wenzl recurrence formula}
We will now consider the last term, $\frac{[n]}{[n+1]}\JW{n} e_{n}
\JW{n}$, in the Wenzl recurrence formula. By expanding this
appropriately, we will see that many of the terms do not
contribute.

Let $P$ denote the leftmost $n-2$ points along the top edge of an
$n$ strand diagram. Define $\mathcal{J}_{n} \subset TL_{n}$ as the
linear span of those diagrams in which any two points of $P$ are
connected together by a strand. This is a left ideal; multiplying
by any diagram on the right does not change this condition.
Further we can write $TL_{n} = \mathcal{J}_{n} \bigoplus
\mathcal{K}_{n}$, where $\mathcal{K}_{n}$ is spanned by the
diagrams in which the points of $P$ are all connected to points on
the bottom edge of the diagram. This collection of diagrams
consists of those diagrams with a single cup at the top right, and
a single cap at some position along the bottom edge, along with
the identity diagram. We denote these diagrams by $g_{n,i}$, with
$i=1,\ldots,n-1$, with the subscript $i$ indicating the position
of the cap. Further, for convenience we write $g_{n,n} = \Id$.
This is illustrated for $n=6$ in Figure
\ref{fig:single-right-cups}. From this, we see $\mathcal{J}_{n}$
has codimension $n$.

\begin{figure}[!hbp]
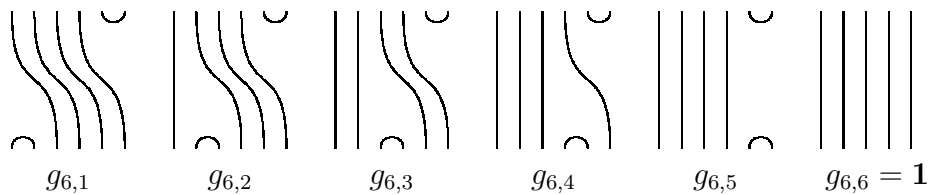

\label{fig:single-right-cups}
\begin{tabular}{cccccc}
\mbox{\xySingleRightCupSixOne} & \mbox{\xySingleRightCupSixTwo} &
\mbox{\xySingleRightCupSixThree} & \mbox{\xySingleRightCupSixFour}
&
\mbox{\xySingleRightCupSixFive} & \mbox{\xySingleRightCupSixSix} \\
$g_{6,1}$ & $g_{6,2}$ & $g_{6,3}$ & $g_{6,4}$ & $g_{6,5}$ & $g_{6,6} =
\Id$
\end{tabular}
\caption{The diagrams spanning $\mathcal{K}_6$.}
\end{figure}

\begin{lem}
The left ideal $\mathcal{J}_{n}$ is contained in the kernel of the
map $TL_{n} \subset TL_{n+1} \To TL_{n+1}$ given by $h \mapsto
\JW{n} e_{n} h$.
\end{lem}
\begin{proof}
If $h$ is a diagram in $\mathcal{J}_{n}$, then we can write $h =
e_i h'$ for some $1 \leq i \leq n-2$, and $h' \in TL_{n}$. Then
$\JW{n} e_{n} e_i = \JW{n} e_i e_{n} = 0$.
\end{proof}

This immediately allows us to simplify the Wenzl recurrence
relation. Write $\JW{n} =
f_{\mathcal{J}}^{(n)}+f_{\mathcal{K}}^{(n)}$, with
$f_{\mathcal{J}}^{(n)} \in \mathcal{J}_n$ and
$f_{\mathcal{K}}^{(n)} \in \mathcal{K}_n$. Then we have
\begin{align*}
 \JW{n} e_n \JW{n} & = \JW{n} e_n (f_{\mathcal{J}}^{(n)}+f_{\mathcal{K}}^{(n)}) \\
    & = \JW{n} e_n f_{\mathcal{K}}^{(n)}.
\end{align*}
Now $\mathcal{K}_n$ is spanned by the diagrams $g_{n,i}$ for
$i=1,\dots,n$, so we can write
\[f_{\mathcal{K}}^{(n)} = \sum_{i=1}^{n} \coeff{n}{g_{n,i}} g_{n,i}.\]

From this we easily obtain
\begin{prop}[Simplified recurrence formula]
The Jones-Wenzl idempotents satisfy
\begin{equation}
 \label{eq:simplified-recurrence-formula}
 \JW{n+1} = \JW{n} \left(\sum_{i=1}^{n} \frac{[n]}{[n+1]} \coeff{n}{g_{n,i}} g_{n+1,i} + g_{n+1,n+1}\right).
\end{equation}

\end{prop}
\begin{proof}
We use the fact that
\begin{equation}
\label{eq:e-g-calculation} e_n g_{n,i} = g_{n+1,i},
\end{equation}
as illustrated in Figure \ref{fig:e-g-calculation}, and
calculate as follows:
\begin{align*}
 \JW{n+1} & = \JW{n} + \frac{[n]}{[n+1]} \JW{n} e_{n} \sum_{i=1}^{n} \coeff{n}{g_{n,i}} g_{n,i} \\
          & = \JW{n} + \frac{[n]}{[n+1]} \JW{n} \sum_{i=1}^{n} \coeff{n}{g_{n,i}} e_n g_{n,i} \\
          & = \JW{n} g_{n+1,n+1} + \frac{[n]}{[n+1]} \JW{n} \sum_{i=1}^{n} \coeff{n}{g_{n,i}} g_{n+1,i} \\
          & = \JW{n} \left(\sum_{i=1}^{n} \frac{[n]}{[n+1]} \coeff{n}{g_{n,i}} g_{n+1,i} + g_{n+1,n+1}\right).
\end{align*}
\end{proof}

\begin{figure}[!hbp]
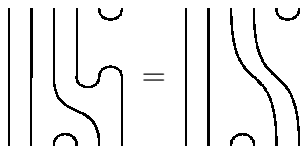

\label{fig:e-g-calculation}
\begin{equation*}
\xyMultiplicativeGeneratorSixFiveTimesGFiveThree = \xyGSixThree
\end{equation*} \caption{A sample calculation, $e_5 g_{5,3} = g_{6,3}$,
illustrating Equation \ref{eq:e-g-calculation}.}
\end{figure}

This simplification of the Wenzl recurrence relation is not in
itself particularly useful. It is still `quadratic' in the sense
that when expanded, each term contains two unknown coefficients.
However, we can now use it to make a direct calculation of the
quantities $\coefficient(g_{n,i})$, which will enable us to
further simplify the recurrence relation to a `linear' form.

\begin{prop}[Further simplified recurrence
formula]\label{prop:further-simplified-recurrence-formula} The
coefficients of the diagrams with `a single right cup' are given
by
\begin{equation}\label{eq:coefficients-single-right-cups}
 \coeff{n}{g_{n,i}} = \frac{[i]}{[n]},
\end{equation}
and the recurrence formula thus becomes
\begin{equation}\label{eq:further-simplified-recurrence-formula}
 \JW{n+1} = \frac{\JW{n}}{[n+1]} \left( \sum_{i=1}^{n+1} [i] g_{n+1,i} \right).
\end{equation}
\end{prop}
\begin{proof}
At $n=1$, there is only one such diagram, $\Id = g_{1,1}$, with
coefficient $1$, as required. Now assume Equation
\ref{eq:coefficients-single-right-cups} holds for some value of
$n$. Equation \ref{eq:further-simplified-recurrence-formula}
follows immediately from Equation
\ref{eq:simplified-recurrence-formula}, by the following
calculation:
\begin{align*}
 \JW{n+1} & = \JW{n} \left(\sum_{i=1}^{n} \frac{[n]}{[n+1]} \frac{[i]}{[n]} g_{n+1,i} + g_{n+1,n+1}\right) \\
          & = \JW{n} \left(\sum_{i=1}^{n} \frac{[i]}{[n+1]} g_{n+1,i} + \frac{[n+1]}{[n+1]} g_{n+1,n+1}\right) \\
          & = \frac{\JW{n}}{[n+1]} \left(\sum_{i=1}^{n+1} [i] g_{n+1,i}\right).
\end{align*}
We will now use this to calculate the coefficient of $g_{n+1,i}$
in $\JW{n+1}$. Suppose $h$ is a diagram in $TL_{n}$, and consider
the term $\frac{[i]}{[n+1]} \coefficient(h) h g_{n+1,i}$ on the
right hand side of Equation
\ref{eq:further-simplified-recurrence-formula}. We will determine
the diagrams $h$ and values of $i$ for which this term contributes
to the $g_{n+1,j}$ term in $\JW{n+1}$. There are several cases to
consider.
\begin{enumerate}
\item The diagram $h$ contains a cap connecting two of the leftmost $n-1$ points at the bottom of the
diagram. In this case $h g_{n+1,i}$ has $n-4$ or fewer through
strands, and so can not contribute to the $g_{n+1,j}$ term in
$\JW{n+1}$. An example of this appears in Figure \ref{fig:cases}.
\item There is no such cap in $h$, but there is a cap connected the rightmost two points at the bottom of the
diagram. In this case the diagram $h g_{n+1,i}$ has a vertical
strand on the right hand side, and so again can not contribute. An
example appears in Figure \ref{fig:cases}.
\begin{figure}[!hbp]
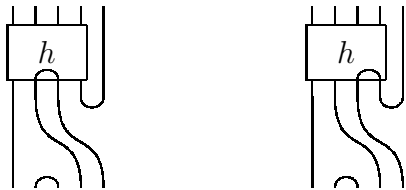
\label{fig:cases} %% perhaps don't let this float away?
\begin{equation*}
\xyCaseOne \qquad\qquad\qquad \xyCaseTwo
\end{equation*}
\caption{Examples illustrating the first two cases in Proposition
\ref{prop:further-simplified-recurrence-formula}. }
\end{figure}
\item There are no such caps, and $h$ is the identity diagram. In this case \[\frac{[i]}{[n+1]} \coeff{n}{h} h
g_{n+1,i} = \frac{[i]}{[n+1]} g_{n+1,i}.\]
\end{enumerate}
These cases are exhaustive, and so it is easily seen that there is
exactly one contribution to the $g_{n+1,j}$ term in $\JW{n+1}$,
coming from the identity term in $\JW{n}$ and the $g_{n+1,j}$ term
of the summation, and so the coefficient of $g_{n+1,j}$ in
$\JW{n+1}$ is exactly $\frac{[j]}{[n+1]}$. Thus by induction the
claimed result holds for all values of $n$.
\end{proof}

\begin{rem}
An analogue of this `linear' recurrence relation for idempotents in
the $\mathfrak{sl}_3$ spider (c.f. \cite{MR1403861}) appears in
Dongseok's work \cite{MR2360947,math.QA/0310143}, where it is
called a `single clasp expansion'.
\end{rem}

\section{Unfolding the recurrence formula}
Let's now think about the map $(\textrm{diagram}) \mapsto
(\textrm{diagram}) \sum_{i=1}^{n+1} \frac{[i]}{[n+1]} g_{n+1,i}$.
Multiplying an $n$ strand diagram by $g_{n+1,i}$ can be thought of
as `inserting a cap at the $i$-th position, and folding up the
right strand':

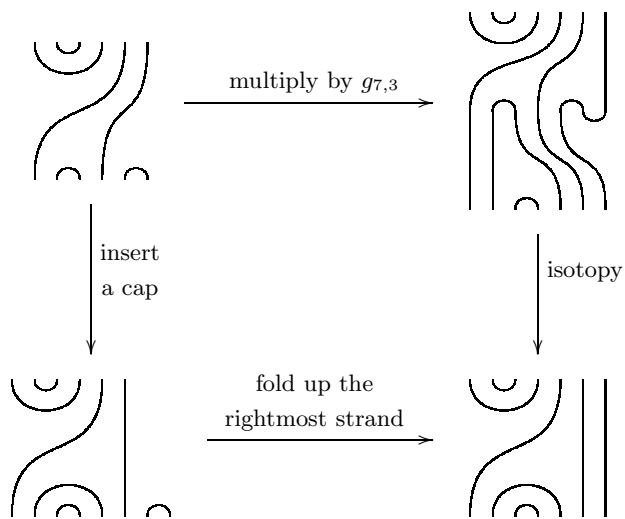
\begin{equation*}
\xymatrix@C=85pt@R=45pt@M+=10pt{
 \xyTLSixSeventyNine \ar[r]^{\txt{\scriptsize{multiply by $g_{7,3}$}}}  \ar[d]^{\txt{\scriptsize{insert}\\\scriptsize{a cap}}}&
                             \xyTLSixSeventyNineTimesGSevenThree \ar[d]^{\txt{\scriptsize{isotopy}}}\\
 \xyTLSixEightyOneAndCap \ar[r]^{\txt{\scriptsize{fold up the}\\\scriptsize{rightmost
 strand}}} & \xyTLSevenTwoHundredAndSixtyFive
}
\end{equation*}

Each diagrammatic term in $\JW{n+1}$ thus arises from a sum of
contributions generated in this way. Choose some diagram $D$ in
$TL_{n+1}$. To determine which terms in $\JW{n}$ contribute to the
coefficient of $D$ in $\JW{n+1}$, we should take $D$, and `fold
down the right strand, then select and remove an innermost cap'.
It is only the terms in $\JW{n}$ involving these diagrams which
matter in calculating the coefficient of $D$ in $\JW{n+1}$.
Suppose we chose to remove an innermost cap at position $i$. The
resulting diagram, when multiplied by the $g_{n+1,i}$, gives the
original diagram $D$.

\begin{prop}
Suppose $D$ is a diagram in $TL_{n+1}$. Let $\hat{D} \in
TL_{n,n+2}$ be the diagram obtained by folding down the top right
end point of $D$. Let $\{i\}$ be the set of positions of innermost
caps in $\hat{D}$, and $D_i \in TL_n$ be the diagram obtained by
removing that innermost cap. Then
\begin{equation}\label{eq:reduction-formula}
    \coeff{n+1}{D} = \sum_{\{i\}} \frac{[i]}{[n+1]}
    \coeff{n}{D_i}.
\end{equation}
\end{prop}

\begin{example}
Consider the diagram $\xyTLFiveEight \in TL_5$. Folding down the
rightmost strand gives $\xyFoldDownTLFiveEight$. There are now two
innermost caps we can remove, at positions $2$ and $5$. Thus
\begin{equation*}
\coeff{5}{\xyTLFiveEight} = \frac{[2]}{[5]}\coeff{4}{\xyTLFourTwo}
+ \frac{[5]}{[5]} \coeff{4}{\xyTLFourThree}.
\end{equation*}
We can continue in this way. The diagram $\xyTLFourTwo$ folds down
to give $\xyFoldDownTLFourTwo$, with only one cap to remove, and
similarly $\xyTLFourThree$ folds down to $\xyFoldDownTLFourThree$.
Thus

%% what's broken in this piece of code?
% \begin{align}
% \xybox{(0,-4.)*\xybox{(1.5,8);(1.5,0) **\crv{(2.5,2.4)&(0.5,5.6)},
% }}
% \end{align}

\begin{equation*}
 \coeff{5}{\xyTLFiveEight} =
 \frac{[2][3]}{[5][4]}
      \coeff{3}{\xyTLThreeOne} + \frac{[5][2]}{[5][4]}
      \coeff{3}{\xyTLThreeOne}
    = \frac{[2][3]+[5][2]}{[5][4]}.
\end{equation*}
\end{example}

%% do another example?

Thus the coefficient of a diagram is a certain sum over sequences
of choices of arcs to remove. Iterating the calculation in
Equation \ref{eq:reduction-formula} allows us to find the
coefficient of any diagram. Although this calculation is based on
a recursive step, it is very different from Wenzl's formula in
Equation \ref{eq:wenzl-formula}. In particular, we never need to
perform any multiplications in the Temperley-Lieb algebra, and we
can find the coefficient of a diagram without calculating the
entire projection, by performing simple combinatorial operations
on the diagrams.

\section{An explicit formula}
It is possible to write down an explicit formula giving the result
of this calculation, but it is made somewhat awkward by the fact
that the numbering of the strands changes as we successively
remove innermost caps.

A good way to think about the diagrams is as a `capform'
\cite{MR1858113}, produced by `folding the diagram down to the
right'.
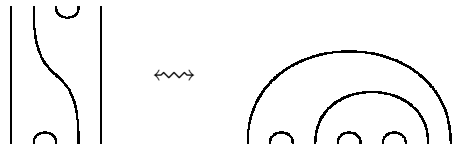
\begin{equation*}
\xyTLFiveEightFullSize \quad \leftrightsquigarrow \quad
\xyCapFormTLFiveEight
\end{equation*}

Now, for a diagram with $n$ strands, let
\begin{equation*}
S = \setc{(s_1,\ldots,s_n)\in \Natural^n}{\parbox{200pt}{the $s_i$
are all distinct, $1 \leq s_i \leq n+i-1$, $s_i$ is the position
of the left end of a cap for each $i$, and if $\tilde{s_i}$
denotes the position of the corresponding right end, then if
$i<j$, and $s_i < s_j$, then $\tilde{s_i} < s_j$ also}}.
\end{equation*}
The sequences in $S$ specify choices of orders in which to remove
strands. The restriction $1 \leq s_i \leq n+i-1$ ensures that we
only remove a strand when its initial point is in the left half of
the capform, and the second restriction ensures that we remove
only innermost caps.

This set $S$ is not quite what is needed, because although it
describes the orders in which we can remove strands, the factors
appearing in Equation \ref{eq:reduction-formula} depend on the
position of the cap at the moment we remove it.

\noop{ This position is given by the map $\tau:(s_1,\ldots,s_n)
\mapsto (t_1,\ldots,t_n)$
\begin{align*}
t_1 & = s_1, \\
t_2 & = \begin{cases}s_2 & \text{if $s_1 > s_2$} \\ s_2 - 2 &
\text{if $s_1 < s_2$}\end{cases} \\
 t_i & = s_i - 2 \kappa_i, \quad \text{where $\kappa_i =
\#\setc{1 \leq j \leq i-1}{s_j < s_i}$}
\end{align*} }

This position is given by the map $\tau:s \mapsto s - \kappa(s)$,
where $$\kappa(s)_i = \#\setc{1 \leq j \leq i-1}{s_j < s_i}.$$
Thus for example $\tau(s)_2 = \begin{cases}s_2 & \text{if $s_1 >
s_2$} \\ s_2 - 2 & \text{if $s_1 < s_2$}\end{cases}$.

Then we have
\begin{prop}\label{prop:formula}
The coefficient in $\JW{n}$ of a diagram $D$ with index set $S$,
as given above, is
\begin{equation}
\coeff{n}{D} = \frac{1}{[n]!} \sum_{s \in S} [\tau(s)],
\end{equation}
using the convenient notations $[n]!=[n][n-1]\cdots[1]$ and
$[(t_1,\ldots,t_n)]=[t_1]\cdots[t_n]$.
\end{prop}

\begin{example}
We redo the calculation of $\coeff{5}{\xyTLFiveEight}$. The index
set has two elements, $S=\{(2,5,7,4,1),(5,2,7,4,1\}$. Then
$\tau(S) = \{(2,3,3,2,1),(5,2,3,2,1)\}$, and so
\begin{equation*}
\coeff{5}{\xyTLFiveEight} =
\frac{[2][3][3][2][1]+[5][2][3][2][1]}{[5]!} =
\frac{[2][3]+[5][2]}{[5][4]},
\end{equation*}
as we calculated before.
\end{example}

\section{\texorpdfstring{$k$}{k}-moves}\label{sec:kmoves}
We'll next apply this algorithm for computing coefficients to
prove `$k$-move invariance'. A $k$-move acts on the capform of a
diagram transforming a collection of $k$ nested caps with centre
strictly in the left half of the capform into $k-1$ nested caps to
the right of a single cap, while leaving the rest of the diagram
unchanged. We apply $k$-moves to rectangular Temperley-Lieb
diagrams by converting to a capform, applying the move as
described, and converting back.

Thus, a valid $4$-move is illustrated below.
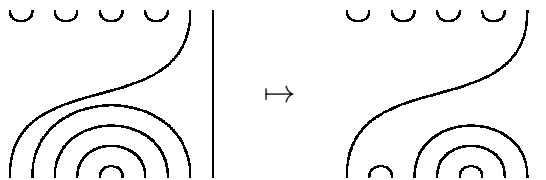
\begin{equation*}
\xyTLMoveInitial \quad \mapsto \quad \xyTLMoveFinalOne
\end{equation*}

The condition that the centre of the capform must lie in the left
half of the diagram requires that the move does not decrease the
number of through strands in the original diagram.

The following theorem relating the coefficients of diagrams
obtained by $k$-moves allows very efficient calculations in many
situations.

\begin{prop}\label{prop:k-move}
If $D'$ is obtained from a diagram $D \in TL_n $ by a $k$-move
then
\begin{equation*}
[k] \coeff{n}{D} = \coeff{n}{D'}.
\end{equation*}
\end{prop}

The proof is a somewhat complicated combinatorial argument, based
on the algorithm above, and manipulation of relations amongst the
quantum integers.

We use the notation of Proposition \ref{prop:formula}. First we
describe the structure of the index set $S'$ for the diagram $D'$,
in terms of the index set $S$ for $D$.

Each $s \in S$ describes an order in which to successively remove
strands. In particular, it tells us the (increasing) times at
which we remove each of the $k$ nested caps. Associated to this
ordering we have several possible orderings for the diagram $D'$.
Instead of removing the $k$ caps in order, we can now remove the
additional single cap at any point instead. Thus we obtain $k$
different elements of $S'$, which remove strands in the rest of
the diagram at exactly the same times as $s$. At some point
(different for each of the $k$ elements) instead of removing the
current innermost cap of the $k$ nested caps, we remove the new
single cap. It is not too hard to see that we obtain all valid
sequences in $S'$ this way, and each exactly once. This is
formalised in the next paragraph.

Suppose the leftmost arc of the $k$ nested caps in $D$ is the
$a$-th strand. For each $s \in S$, define $s^{(1)},\ldots,s^{(k)}
\in S'$ as follows. Let $j_1 < \cdots < j_k$ be the positions in
$s$ of the numbers $a+k-1, \ldots, a$, and call these positions
`marked'. Because of the nested structure, we have $s_{j_i} =
a+k-i$. In the following we'll often need to describe the elements
of a sequence of the marked positions, so we'll introduce the
following notation: $$\db{s} = \left(s_{j_i}\right)_{i=1}^k =
(a+k-1,a+k-2 ,\ldots,a).$$ Now let $s^{(i)}$ be the same as $s$ in
the unmarked positions, and $$\db{s^{(i)}} =
(a+k,a+k-1,\ldots,a+k+2-i,\underbrace{a}_{\text{$i$-th position}}
,a+k-i+1,\ldots,a+3,a+2).$$ That is, $\db{s^{(i)}} = \db{s} +
(1,1,1,\ldots,1,i-k,2,\ldots,2,2)$.

\begin{lem}
$$S' = \setc{s^{(i)}}{s \in S, i \in 1,\ldots,k}$$ and so
$$\coeff{n}{D'} = \frac{1}{[n]!} \sum_{s \in S} \sum_{i=1}^{k}
[\tau(s^{(i)})].$$
\end{lem}

\begin{proof}[Proof of Proposition \ref{prop:k-move}]
We calculate $\tau(s^{(i)})$, then prove that $\sum_{i=1}^k
[\tau(s^{(i)})] = [k] [\tau(s)]$.

Firstly, suppose $\db{\kappa(s)} = (\kappa_1, \ldots, \kappa_k)$,
so $\tau(s)_{j_i} = a+k-i-2\kappa_i$. For brevity we'll define
$b_i = a+k-i-2\kappa_i$. Outside the marked positions,
$\kappa(s^{(i)})$ agrees with $\kappa(s)$, and
$\db{\kappa(s^{(i)})} = (\kappa_1, \kappa_2, \ldots, \kappa_{i-1},
\kappa_i, \kappa_{i+1}+1, \ldots, \kappa_k + 1)$. Thus
\begin{align*}
\db{\tau(s^{(i)})} & = (a+k-2\kappa_1, a+k-1-2\kappa_2, \ldots,
    a+k-(i-2)-2\kappa_{i-1}, a- 2 \kappa_i,
    a+k-(i+1)-2\kappa_{i+1},\ldots,a-2\kappa_k)\\
        & = (b_1 + 1, b_2 +1, \ldots, b_{i-1}+1, b_i-k+i,
        b_{i+1},\ldots,b_k).
\end{align*}
%% learn how to use split!!

We want to prove that $\sum_{i=1}^k [\db{\tau(s^{(i)})}] =
[k][(b_1,\ldots,b_k)]$. To this end, define the partial sum $T_l =
\sum_{i=1}^l [\db{\tau(s^{(i)})}]$. We will show that
\begin{equation}\label{eq:T-formula}
 T_k = T_l + \prod_{j=1}^l [b_j + 1] \cdot [k-l] \cdot \prod_{j=l+1}^{k} [b_j]
\end{equation}
for each $l$, and so, evaluating at $l=0$, $T_k = [k]
\prod_{j=1}^{k} [b_j]$, as required.

Certainly Equation \ref{eq:T-formula} holds for $l=k$, since $[0]
= 0$. Suppose it holds for some value $l$. We can pull out the
final term of the summation, and obtain
\begin{align*}
  T_k & = T_l + \prod_{j=1}^l [b_j + 1] \cdot [k-l] \cdot \prod_{j=l+1}^{k} [b_j]\\
      & = T_{l-1} + \prod_{j=1}^{l-1} [b_j+1] \cdot ([b_l - k + l]
            + [b_l + 1] [k-l]) \prod_{j=l+1}^{k} [b_j]\\
\intertext{and by Lemma \ref{lem:quantum-integer-lemma}, this is}
    & = T_{l-1} + \prod_{j=1}^{l-1} [b_j+1] \cdot [k-l+1][b_l]
        \cdot \prod_{j=l+1}^{k} [b_j].
\end{align*}
Thus Equation \ref{eq:T-formula} also holds for $l-1$,
establishing the result.
\end{proof}

%\section{Results of Khovanov and Frenkel}
%% write this!

\section{Results of Ocneanu and of Reznikoff}
A similar formula has previously been published for these
coefficients, by Ocneanu \cite{MR1907188}, although a proof of
that formula was not given. His formula uses the alternative
convention that closed loops have value $[2]$.

%% comparison of the complexity of the formulas

Subsequently, a proof of special cases of this formula was been
provided by Reznikoff \cite{rez:rptla,MR2375712}. The proof
confirms Ocneanu's formula for diagrams in $TL_n$ with $n-2$ or
$n-4$ through strings, and uses very different methods (via the
Brauer representation of the Temperley-Lieb algebra) from those
employed here.

The method presented here readily reproduces Reznikoff's results. Some
examples of this are given below. In doing so, this proves that Ocneanu's
formula and the formula here are equivalent for diagrams with $n-2$ or
$n-4$ through strings. However, I have been unable to obtain a direct proof
that the formulas agree for all diagrams.

It is reasonably easy to prove that in limited cases the $k$-move
invariance described in \S \ref{sec:kmoves} holds for Ocneanu's
formula as well. In particular, for two diagrams related by a
$k$-move that involves no through strings at all, the coefficients
given by Ocneanu's formula agree with Proposition \ref{prop:k-move}.

This suggests a way to prove the equivalence of the formula here and
Ocneanu's directly. If we knew the two formulas agreed on some class
of simple diagrams, they would also agree on all diagrams obtained
from these by a sequence of $k$-moves and inverse $k$-moves.
However, the equivalence classes of diagrams under these moves are
not particularly large; they each contain a single diagram with no
nested caps or cups. %% what's really going on?

\section{An application to diagrams with \texorpdfstring{$n-4$}{n-4} through strings}
In this section, we give an explicit calculation of the
coefficient of certain diagrams with exactly $n-4$ through strings.
Although we only do one case here, all the other cases are no more
difficult.

We use a combination of the summation formula of Equation
\ref{eq:reduction-formula} and the $k$-moves of the \S \ref{sec:kmoves}.
Hopefully this will illustrate the computational power of these techniques!

A diagram with $n-4$ through strings has exactly $2$ caps and $2$
cups. We restrict our attention to those diagrams with no nested
caps or cups. Consider such a diagram $D$. Thus we can
unambiguously refer to these as the `left cap', `right cap', `left
cup', and `right cup'. Suppose the leftmost points of these arcs
occur at positions $b_1$, $b_2$, $t_1$ and $t_2$. (And of course,
$b_2 \geq b_1 + 2$, $t_2 \geq t_1 + 2$.)

Because the coefficients of diagrams are preserved when the diagram is
reflected in a horizontal line, we may assume that the right cap is no
further to the right than the right cup, that is, that $t_2 \geq b_2$.

In this configuration, we can apply an inverse $(n - t_2)$-move, moving to
right cup as far to the right as possible, obtaining the diagram $D'$, with
$t_2 = n-1$. The coefficients are related by $\coeff{n}{D'} =
\frac{1}{[n-t_2]} \coeff{n}{D}$, by Proposition \ref{prop:k-move}.

We now apply the reduction formula. Folding down the top right
point of the diagram turns the right cup into a through strand.
Next, we have to choose one of the caps, at positions $b_1$ and
$b-2$, to remove. The resulting diagrams are $D_{b_1}'$, with a
cap at position $b_2 - 2$ and a cup at position $t_1$, and
$D_{b_2}'$ with a cap at position $b_1$ and a cup at position
$t_1$. Then Equation \ref{eq:reduction-formula} then tells us
\begin{equation}
\coeff{n}{D'} = \frac{[b_1]}{[n]} \coeff{n-1}{D_{b_1}'}
          + \frac{[b_2]}{[n]} \coeff{n-1}{D_{b-2}'}.
\end{equation}
The coefficients appearing here depend on the relative ordering of
$b_2 - 2$ and $t_1$ (for the first term), and of $b_1$ and $t_1$
(for the second term). We'll assume now that $t_1 \geq b_2 - 2$.
(The other two cases, $b_1 \leq t_1 \leq b_2 - 2$ and $t_1 \leq
b_1$, are exactly analogous.) In this case, we can apply an
inverse move to each diagram, as above, to move the cup to the far
right, and then use Equation
\ref{eq:coefficients-single-right-cups}. Thus
\begin{align*}
\coeff{n-1}{D_{b-1}'} & = \frac{[b_2 - 2][n-1-t_1]}{[n-1]} \\
\coeff{n-1}{D_{b-2}'} & = \frac{[b_1][n-1-t_1]}{[n-1]}.
\end{align*}
Putting this all together, we obtain
\begin{align*}
 \coeff{n}{D} & = \frac{[b_1]([b_2]+[b_2-2])[n-1-t_1][n-t_2]}{[n][n-1]} \\
              & = \frac{[2][b_1][b_2-1][n-1-t_1][n-t_2]}{[n][n-1]}
\end{align*}
This agrees with the formula given as Equation 3 in
\cite{MR2375712}, for `Style 3' diagrams. (The other two styles
of diagrams there correspond exactly to the other two cases
described previously.)

% ----------------------------------------------------------------
%\newpage
\bibliographystyle{alpha}
\bibliography{../../../bibliography/bibliography}
\end{document}